\documentclass[review,sort&compress]{elsarticle}

\usepackage{hyperref}
\usepackage{amssymb}
\usepackage{latexsym,amsmath}
\usepackage{enumerate}
\usepackage{enumitem}
\usepackage{amsthm}
\newtheorem{theorem}{Theorem}[section]
\newtheorem{lemma}[theorem]{Lemma}
\newtheorem{proposition}[theorem]{Proposition}
\newtheorem{definition}[theorem]{Definition}
\newtheorem{remark}[theorem]{Remark}

\newproof{pf}{Proof}
\newcommand{\Int}{\displaystyle \int}

\newcommand{\Frac}{\displaystyle \frac}

\newcommand{\Lim}{\displaystyle \lim}
\newcommand{\R}{\mathbb R}

\numberwithin{equation}{section}
\allowdisplaybreaks
\journal{Applied Mathematics and Computation}
\makeatletter
\def\ps@pprintTitle{%
  \let\@oddhead\@empty
  \let\@evenhead\@empty
  \def\@oddfoot{\reset@font\hfil\thepage\hfil}
  \let\@evenfoot\@oddfoot
}
\makeatother

\begin{document}

\begin{frontmatter}

\title{\textbf{Integrable solutions of a generalized mixed-type functional integral equation}}


\author[HW]{Haydar Abdel Hamid\corref{cor1}}
\cortext[cor1]{Corresponding author.}
\ead{habdelhamid@fbsu.edu.sa}
\author[HW]{Waad Al Sayed}
\ead{wsaid@fbsu.edu.sa}
\address[HW]{Faculty of Sciences and Humanities, Fahad Bin Sultan
University, P.O. Box 15700 Tabuk 71454, Saudi Arabia.}

\begin{abstract}
In this work, we prove the existence of  integrable solutions for the following  generalized mixed-type nonlinear functional integral equation
$$x(t)=g\left(t,(Tx)(t)\right)+f\left(t,\int_0^t k(t,s)u(t,s,(Qx)(s))\;ds\right),\;t\in[0,\infty).$$
Our result is established by means of a Krasnosel'skii type fixed point theorem proved in  
[M.A. Taoudi: \textit{Integrable solutions of a nonlinear functional integral equation 
on an unbounded interval}, Nonlinear Anal. 71 (2009) 4131-4136]. In the last
section we give an example to illustrate our result.
\end{abstract}


\end{frontmatter}


\section{Introduction}
Consider  the following  mixed-type nonlinear functional integral equation

\begin{equation}
\label{e1}x(t)=g\left(t,(Tx)(t)\right)+f\left(t,\int_0^t k(t,s)u(t,s,(Qx)(s))\;ds\right),\;t\in[0,\infty).
\end{equation}
where $f$, $g:$ $\mathbb {R^+}\times \mathbb {R} \rightarrow \mathbb {R}$, 
$k :$ $\mathbb {R^+}\times \mathbb {R^+} \rightarrow \mathbb {R}$, 
$u:$ $\mathbb {R^+}\times \mathbb {R^+} \times \mathbb {R} \rightarrow \mathbb {R}$ 
and $(Tx)(t)$, $(Qx)(t)$ are given while $x(t)$ is an unknown function.

In \cite{BaCh}, the authors studied the existence of integrable solutions of the 
following  special case of the equation  (\ref{e1})
 
$$x(t)=f\left(t,\int_0^t k(t,s)u(s,x(s))\;ds\right),\;t\in[0,\infty).$$

The following generalization of this equation

$$x(t)=g(t,x(t))+f\left(t,\int_0^t k(t,s)u(s,x(s))\;ds\right),\;t\in[0,\infty),$$
has been studied by \cite{Ta09} with the presence of the perturbation term $g$.

In this paper, we are going  to study the existence of integrable solutions of 
the more general form (\ref{e1}). A classical point of view for solving Eq.  (\ref{e1}) is to write the equation in the form
 
\begin{equation}
\label{e2} Ax+Bx=x,
\end{equation}
where $A$ and $B$ are two nonlinear operators. 

Fixed point theory seems to be one of the most natural  and powerful  tools 
in studying the solvability of integral equations in the form (\ref{e2}). In \cite{Kr58},  Krasnosel'skii established 
 a fixed point theorem which was frequently used to solve some special integral equations in the form (\ref{e2}), see \cite{Kr64,KZPS}.  Krasnosel'skii combined the 
famous Banach contraction principle  of \cite{Ba22} and the classical Schauder 
fixed point theorem of \cite{Sch30} to prove that $A+B$ has a fixed point in a
nonempty closed convex subset $\mathcal{M}$ of a real Banach space $X$ if $A$ and $B$ satisfy
the following conditions (see \cite{Kr58,Sm80}):

\begin{itemize}[label=$\bullet$]
\setlength\itemsep{0em}
	\item $A$ is continuous and compact;
	\item $B$ is a strict contraction;
\item $A\mathcal{M}+B\mathcal{M}\subseteq \mathcal{M}$.
\end{itemize}

Generalizations and improvements of such a result have been made in several directions,
we refer for example to the papers \cite{ART10,Bar03,BarTe05,Ga10,Ga09,LaTa07,LaTaZe06,Re96,BuKi98,GaLaMoTa12,Ta10} and the references therein. 
A deep variant of Krasnosel'skii type fixed point theorems is established by Latrach and Taoudi 
in \cite{LaTa07}. In \cite{Ta09},  Taoudi established an improvement of this variant.
The most important advantage of \cite[Theorem 2.1]{LaTa07}  and \cite[Theorem 3.7]{Ta09} is  that the operator $A$ is not assumed to be compact.
Let us  recall the Krasnosel'skii type fixed point theorem
 \cite[Theorem 3.7]{Ta09}.

\begin{theorem}\label{theow2}
Let $\cal M$ be a nonempty bounded closed convex subset of a Banach space $X$. Suppose that $A:\mathcal{M}\rightarrow X$ and $B:\mathcal{M}\rightarrow X$ such that:\\
{\rm{(i)}} $A$ is (ws)-compact;\\
{\rm{(ii)}} There exists $\gamma \in \left[0,1 \right[ $ such that $\mu (AS+BS) \leq \gamma \mu(S)$ for all $S\subseteq \cal {M}$; here $\mu$ is an arbitrary measure of weak noncompactness on $X$;\\
{\rm{(iii)}} $B$ is a separate contraction;\\
{\rm{(iv)}} $A{\cal{M}} + B \cal{M}\subseteq \cal {M}$. \\
Then there is $x\in \cal {M}$ such that $Ax+Bx=x$ 
\end{theorem}
Our aim is to prove the existence of solutions of Eq. (\ref{e1}) 
in the space $L^{1}(\mathbb{R}_{+})$ of  Lebesgue integrable functions on the 
the real half-axis $\mathbb{R}_{+}=[0,\infty)$.  Theorem \ref{theow2} plays a crucial role 
in establishing our result in Theorem \ref{mth1}. A technique of measures of
weak noncompactness used in \cite{BaCh} will be implemented.
The result obtained in this paper generalizes  the result of \cite{Ta09} to
a more general equation such as Eq. (\ref{e1})  and extends the technique used 
in \cite{BaCh} to our more general context under special assumptions.

The outline of this paper is as follows. In Section \ref{stwo}, we recall some notations, definitions and
basic tools which will be used in our investigations. Section \ref{sthr} is 
devoted to state our main result and to prove some preliminary results. In Section \ref{sfo},
we prove our main result. In the last section we construct a nontrivial example 
illustrating our result.
\section{Preliminaries}\label{stwo}

In this section we recall without proofs some of useful  facts on Lebesgue space
$L^{1}(\mathbb{R}_{+})$, the superposition operator, contractions, ($ws$)-compact operators and  measures of weak noncompactness.
\subsection{The Lebesgue Space}

Let $\mathbb{R}$ be the set of real numbers and let $\mathbb{R}_{+}$ be the interval $[0,+\infty)$.
For a fixed Lebesgue measurable subset $I$ of $\mathbb{R}$, let $meas(I)$
be the Lebesgue measure of $I$ and denote by $L^{1}(I)$ the space of Lebesgue
integrable functions on $I$, equipped with the standard norm
 
 $$\|x\|_{I}=\|x\|_{L^{1}(I)}=\int_I|x(t)|\;dt.$$
In the case when $I=\mathbb{R}_{+}$ the norm $\|x\|_{L^{1}(\mathbb{R}_{+})}$
will be briefly denoted by $\|x\|$. Now, let us recall the following criterion 
of weak noncompactness in the space $L^{1}(\mathbb{R}_{+})$ established by Dieudonne 
\cite{dieud}. It will be frequently used in our discussions. 
\begin{theorem}\label{bs}
A bounded set $X$ is relatively weakly compact in $L^1(\mathbb{R}_{+})$ if and only if the following two conditions are satisfied:\\
{\rm{(a)}} for any $\epsilon >0$ there exists $\delta >0$ such that if $meas (D)\leq \delta$ then $\Int_D \left| x(t)\right| dt \leq \epsilon$ for all $x\in X$.\\
{\rm{(b)}} for any $\epsilon >0$ there exists $\tau >0$ such that $\Int_\tau^\infty \left| x(t)\right| dt \leq \epsilon$ for any $x\in X$. 
\end{theorem} 
\subsection{The Superposition Operator}

For a fixed interval $I\subset\mathbb{R}$, bounded or not, consider a function
$f:I\times\mathbb{R}\rightarrow\mathbb{R}$. The function $f=f(t,x)$ is said to satisfy
the Carath\'eodory conditions if it is  Lebesgue measurable in $t$ for every fixed
 $x\in\mathbb{R}$ and continuous in $x$ for almost every $t\in I$. 
The following theorem due to Scorza Dragoni \cite{SD} explains the structure of functions 
satisfying Carath\'eodory conditions.
\begin{theorem}\label{Cara}
Let $I$ be a bounded interval and let $f:I\times\mathbb{R}\longrightarrow\mathbb{R}$
be a function satisfying Carath\'eodory conditions. Then, for each $\epsilon>$ there exists a closed
subset $D_{\epsilon}$ of the interval $I$ such that $meas(I\backslash D_{\epsilon})\leq\epsilon$
and $f|_{D_{\epsilon}\times\mathbb{R}}$ is continuous.
\end{theorem}
The \textit{superposition 
operator} (or \textit{Nemytskii operator}) associated with $f$ is defined as follows.
\begin{definition}\label{sod}
Let $f:I\times\mathbb{R}\rightarrow\mathbb{R}$ be a Carath\'eodory function. 
The superposition operator generated by $f$ is the operator $N_{f}$ which assigns
to each real measurable function on $I$ the real function $(N_{f}x)(t)=f(t,x(t))$, $t\in I$.
\end{definition}
Appell and Zabrejko \cite{ApZa} realized a thorough study of superposition 
operators in which several types of results are presented. Among them, we recall 
the following fundamental theorem which gives a necessary and sufficient condition 
ensuring that $N_{f}$ maps continuously  $L^{1}(I)$ into itself when $I$ is an unbounded interval.
The authors generalized the result proved by Krasnosel'skii \cite{Kr} when $I$ is a bounded interval  
\begin{theorem}\label{sup}
The superposition operator $N_f$ generated by the function $f$ maps the space $L^1(I)$ continuously into itself if
and only if

$$|f(t,x)|\leq a(t)+b|x|,$$
for all $t\in I$ and all $x\in\mathbb{R}$, where $a\in L^1(I)$ and $b\geq 0$ is a constant.
\end{theorem}
\subsection{Contractions}
Let  $(X,d)$  be a metric space. It  is well known that a mapping $B:X\rightarrow X$
is called a \textit{strict contraction} if there exists $k\in(0,1)$ such that $d(Bx,By)\leq k\,d(x,y)$
for every $x,y\in X$. In the following definition, 
we recall the notion of a \textit{large contraction} introduced by Burton in \cite{BU96}.
\begin{definition}
Let $(X,d)$ be a metric space. We say that $B:X\rightarrow X$ is a large contraction 
if $d(Bx,By)< d(x,y)$ for $x$, $y\in X$ with $x\not = y$ and if $\forall \epsilon >0$ 
there exists $\delta <1$ such that $\left[ x,y\in X, d(x,y)\geq \epsilon \right] \Rightarrow d(Bx,By)\leq \delta d(x,y)$.
\end{definition}
Now, we recall the notion of a \textit{separate contraction} introduced by Liu and Li  \cite{LL06}
which is weaker than the \textit{strict contraction} and \textit{large contraction} in the sense that 
every \textit{strict contraction} is a \textit{separate contraction} and every \textit{large contraction} is a \textit{separate contraction}.
The same authors  gave in \cite{LL08} an example of a separate contraction which is not a strict contraction
and another example of a separate contraction which is not a large contraction.

\begin{definition}
Let $(X,d)$ be a metric space. We say that $B:X\rightarrow X$ is a separate contraction if there exist two functions $\varphi$, $\psi: \mathbb{R^+} \rightarrow \mathbb {R^+}$ satisfying:\\
{\rm{(1)}} $\psi(0)=0$, $\psi$ is strictly increasing,\\
{\rm{(2)}} $d(Bx,By)\leq \varphi (d(x,y))$,\\
{\rm{(3)}} $\psi (r)+ \varphi(r)\leq r$ for $r>0$.
\end{definition}
\subsection{(\textit{ws})-compact Operators} 

We recall the following definition  from \cite{JJ}.
\begin{definition}
 Let $M$ be a subset of a Banach space $X$. A continuous map $A:M\rightarrow X$ is said to be (ws)-compact if for any weakly convergent sequence $(x_n)_{n\in \mathbb N}$ in $M$ the sequence  $(Ax_n)_{n\in \mathbb N}$  has a strongly convergent subsequence in $X$. 
\end{definition}
From this definition it immediately follows  that a map $A$ is ($ws$)-compact if and only if it maps relatively weakly compact sets into relatively compact ones. 
\subsection{Measures of weak noncompactness}
We recall some basic facts concerning measures of weak noncompactness, see \cite{{BaRi}}. Let us assume that $E$ is an infinite
dimensional Banach space with norm $\|.\|$ and zero element $\theta$. Denote by $\mathfrak{M}_E$ the family of all nonempty
and bounded subsets of $E$ and by  $\mathfrak{N}_E^\mathcal{W}$ the subset of  $\mathfrak{M}_E$ consisting of all relatively 
weakly compact sets. For a subset $X$ of $E$, the symbol
${\rm Conv}X$ will denote the convex  hull with respect
to the norm topology. Finally, we denote by $B(x,r)$ the ball centered at $x$ and of radius $r$. We write
$B_r$ instead of $B(\theta,r)$.
\begin{definition}\label{d1} A mapping $\mu:{\mathfrak{M}}_E\rightarrow\mathbb{R}_+$ is
said to be a measure of weak noncompactness in $E$ if it
satisfies the following conditions:\\
{\rm{(1)}} The kernel of $\mu$ defined by ${\rm ker}\mu=\{X\in{\mathfrak{M}}_E:\mu(X)=0\}$ is nonempty and ${\rm
ker}\mu\subset{\mathfrak{N}}_E^\mathcal{W}$.\\ 
{\rm{(2)}} $X\subset Y\Rightarrow \mu(X)\leq\mu(Y)$.\\ 
{\rm{(3)}} $\mu({\rm Conv}X)=\mu(X)$.\\ 
{\rm{(4)}} $\mu(\lambda X+(1-\lambda)Y)\leq\lambda\;\mu(X)+(1-\lambda)\;\mu(Y)$ for
$\lambda\in[0,1]$.\\ 
{\rm{(5)}} If $(X_n)_{n\geq1}$ is a sequence of nonempty, weakly closed subsets of $E$ with $X_{1}$ bounded and  $X_{n+1}\subset X_n$ for $n=1,\;2,\;3,\;...$ and if
$\lim\limits_{n\rightarrow\infty}\mu(X_n)=0$ then $\displaystyle\bigcap_{n=1}^\infty
X_n$ is nonempty.
\end{definition}

The first important  measure of weak noncompactness in a concrete Banach space $E$ was defined by De Blasi \cite{Bl} as follows:

$$\beta(X)=\inf\{\varepsilon>0:{\rm there}\;{\rm is}\;{\rm a}\;{\rm weakly}\;{\rm compact}\;{\rm subset}\;Y\;{\rm of}\;E\;{\rm such}\;
{\rm that}\;X\subset Y+B_\varepsilon\}.$$
The De Blasi measure of weak  noncompactness  $\beta$ plays a significant role in nonlinear
analysis and has many applications,  \cite{BaGo,BaRi,Bl,LaLe}. It is worthwhile
to mention that it is rather difficult to express $\beta$ with a convenient formula
in a concrete Banach space $E$. Our problem under consideration (\ref{e1}) will be studied
in the Banach space $E=L^{1}(\mathbb{R}_{+})$. In  \cite{{BaKn2}}, Bana\'s and Knap
constructed a useful measure of weak noncompactness $\mu$ in the space 
$L^{1}(\mathbb{R}_{+})$. The following construction of $\mu$ is based on the 
criterion of weak noncompactness given in Theorem \ref{bs} due to \cite{dieud}: 
For  a bounded subset $X$ of $L^1(\mathbb{R}_+)$ we  define

\begin{equation}\label{mfe}
\mu(X)=c(X)+d(X),
\end{equation}
where

\begin{equation}\label{cdef}
c(X)=\lim_{\varepsilon\rightarrow 0}\left\{
\sup_{x\in X}\left\{
\sup\left[
\int_\Omega|x(t)|\;dt:\Omega\subset\mathbb{R}_+,\;{\rm meas(\Omega)\leq\varepsilon}
\right]
\right\}
\right\},
\end{equation}
and

\begin{equation}\label{ddef}
d(X)=\lim_{T\rightarrow\infty}\left\{
\sup\left[
\int_T^\infty|x(t)|\;dt:x\in X\right]
\right\}.
\end{equation}

\section{Assumptions, statement of results}\label{sthr}
In this section, we state the existence of solutions  to the functional integral
equation (\ref{e1}) in the space $L^1(\mathbb{R}_+)$.   
First, observe that the problem (\ref{e1}) may be  written in the form

\begin{equation}\label{e1re}
Ax+Bx=x,
\end{equation}
where $B=N_{g}T$ and $A=N_{f}UQ$, where $N_{g}$ and $N_{f}$ are 
the superposition operators associated to $g$ and $f$ respectively (see Definition \ref{sod}) and $U$ 
is the operator defined by 

\begin{equation}\label{U}
(Ux)(t)=\int^{t}_{0}k(t,s)u(t,s,x(s))\;ds.
\end{equation} 
Our aim is to prove that $A+B$ has a fixed point in  $L^1(\mathbb{R}_+)$
by applying Theorem \ref{theow2}. We consider Eq. (\ref{e1}) under the following assumptions:

\begin{enumerate}[label=(A\arabic*)]
\item The functions $g,f:\mathbb{R}_+\times\mathbb{R}\rightarrow\mathbb{R}$ 
satisfies Carath\'eodory conditions and there are  constants  
$b,b_{1}>0$ and  functions  $a, a_{1} \in L^{1}(\mathbb{R}_+)$ such that

$$|g(t, x)| \leq a(t) + b|x|,$$
$$|f(t,x)|\leq a_1(t)+b_1|x|,$$
for $t\in \mathbb{R}_+$ and $x\in \mathbb{R}$. 
\item The operator $Q$ maps continuously the space $L^{1}(\mathbb{R}_{+})$ into itself and there 
are  constants  $\rho_{1},\rho_{2}>0$, 
functions  $\gamma_{1},\gamma_{2} \in L^1(\mathbb{R}_+)$ and 
increasing functions $\phi,\psi: \mathbb{R}_+ \rightarrow \mathbb{R}_+$, 
absolutely continuous such that

$$|(Tx)(t)| \leq \gamma_{1}(t) + \rho_{1}|x(\phi(t))|,$$
$$|(Qx)(t)| \leq \gamma_{2}(t) + \rho_{2}|x(\psi(t))|,$$
for $x\in L^{1}(\mathbb{R}_+)$ and $t\in \mathbb{R}_{+}$. Moreover, there are  constants  $m,M>0$ such that $\phi '(t)\geq m$ and $\psi '(t)\geq M$ for almost all $t\geq 0$. 
\item The mapping $B=N_{g}T$ is a separate contraction.
\item The function $u:\mathbb{R}_+\times\mathbb{R}_+\times\mathbb{R}\rightarrow\mathbb{R}$ 
satisfies Carath\'eodory conditions, i.e., the function $t\rightarrow u(t,s,x)$ is measurable
on $\mathbb{R}_+$ for every fixed $(s,x)\in \mathbb{R}_+\times\mathbb{R}$ and 
the function $(s,x)\rightarrow u(t,s,x)$ is continuous on $\mathbb{R}_+\times\mathbb{R}$ 
for almost every $t\in \mathbb{R}_+$.
\item There are  constants 
$\beta,\lambda>0$,  functions $\alpha,\gamma \in L^1(\mathbb{R}_+)$  
and a  function $h:\mathbb{R}_+\rightarrow\mathbb{R}_+$ which is measurable
with $\Lim_{\delta\rightarrow0}h(\delta)=0$, such that 

\begin{equation}\label{Uin}
\left|u(t,s,x)\right|\leq \alpha(s)+\beta|x|,
\end{equation}
for all $t,s\in \mathbb{R}_+$, $x\in \mathbb{R}$, and 

$$\left|u(t,s,x)-u(t+\delta,s,x)\right|\leq h(\delta)\left[\gamma(s)+\lambda|x|\right],$$
	for any $t, s\in \mathbb{R}_+$, $x\in \mathbb{R}$ and $\delta$ small.
	\item The function $k:\mathbb{R}_+\times\mathbb{R}_+\rightarrow\mathbb{R}$ is measurable such that the linear Fredholm integral operator

$$(Kx)(t)=\int_0^t \left|k(t,s)\right|x(s)\;ds$$
is a continuous map from $L^1(\mathbb{R}_+)$ into itself.
	\item $\gamma=b\,\rho_{1}m^{-1}+b_{1}\rho_{2}\beta M^{-1} \left\|K\right\|<1$, where $\|K\|$ denotes the norm of the operator $K$. 
\end{enumerate}



Now, we are in a position to present our  existence result.
\begin{theorem}\label{mth1} 
Let assumptions {\rm{(A1)-(A7)}} be satisfied. Then, equation {\rm{(\ref{e1})}}
has at least one solution $x\in L^1(\mathbb{R}_+)$. 
\end{theorem}
\begin{remark}
Under the assumption {\rm{(A2)}}, the operators $T$ and $Q$ map $L^{1}(\mathbb{R}_{+})$ into itself.
Moreover, the operator $Q$ is assumed to be continuous. However, the operator $T$ is not assumed to be continuous.
\end{remark}

To prove the main result, we demonstrate the continuity of $A$ 
and we establish  $L^{1}(I)$-estimates for any nonempty measurable subset $I$ of $\mathbb{R}_{+}$.

\begin{proposition}\label{Acont}\rm{
Suppose that assumptions (A1)-(A4), (A6) and  the inequality (\ref{Uin}) from (A5) hold. Then  
\begin{enumerate}[label=(\alph*)]
\item The operator $A$ maps $L^{1}(\mathbb{R}_{+})$ continuously into itself. 
\item For any nonempty measurable subset $I$ of $\mathbb{R}_{+}$ and 
$x\in  L^{1}(\mathbb{R}_{+})$ we have the following estimations

\begin{equation}\label{Aest}
\left\|A x\right\|_{I}\leq \left\|a_{1}\right\|_{I}+b_{1}\left\|K\right\|\left[\left\|\alpha\right\|_{I}+\beta\left\|\gamma_{2}\right\|_{I}+\beta M^{-1}\rho_{2}\left\|x\right\|_{\psi(I)}\right]
\end{equation}
and 

\begin{equation}\label{Best}
\left\|B x\right\|_{I}\leq \left\|a\right\|_{I}+b\left[\left\|\gamma_{1}\right\|_{I}+\rho_{1}m^{-1}\left\|x\right\|_{\phi(I)}\right]
\end{equation}
\end{enumerate}}
\end{proposition}
\begin{proof}[\bf Proof] First we prove that the operator $U$ given by (\ref{U}) maps $L^{1}=L^{1}(\mathbb{R}_{+})$ continuously into itself. 
In view of the inequality (\ref{Uin}) from assumption (A5) and the assumption (A6),  it is easy to observe that
 the operator $U$ transforms $L^{1}$  into itself. Now, let $\left\{x_{n}\right\}$
be a sequence in $L^{1}$ which converges to $x$ in $L^{1}$. We show that 
$\left\{Ux_{n}\right\}$ converges to $Ux$ in $L^{1}$. For every $\tau>0$, 
in view of our assumptions we have
 \begin{eqnarray*} 
\nonumber\left\|Ux_{n}-Ux\right\|_{L^{1}}&\leq& \int^{\tau}_{0}\left|(Ux_{n})(t)-(Ux)(t)\right|\,dt+\int^{\infty}_{\tau}\left|(Ux_{n})(t)\right|\,dt\\
\nonumber\;&&+\int^{\infty}_{\tau}\left|(Ux)(t)\right|\,dt\\
\nonumber\;&\leq& \delta_{n}+\int^{\infty}_{\tau}\int^{t}_{0}\left|k(t,s)\right|\left(\alpha(s)+\beta\left|x_{n}(s)\right|\right)ds\,dt\\
\nonumber\;&&+\int^{\infty}_{\tau}\left|(Ux)(t)\right|\,dt\\
\nonumber\;&=& \delta_{n}+\left\|K\alpha\right\|_{L^{1}([\tau,\infty))}+\beta\left\|Kx_{n}\right\|_{L^{1}([\tau,\infty))}+\left\|Ux\right\|_{L^{1}([\tau,\infty))}
\end{eqnarray*} 
where $\displaystyle\delta_{n}=\int^{\tau}_{0}\left|(Ux_{n})(t)-(Ux)(t)\right|\,dt$.

Now, since $Ux\in L^{1}$ and $K\alpha\in L^{1}$ we deduce that terms $\left\|K\alpha\right\|_{L^{1}([\tau,\infty))}$ and $\left\|Ux\right\|_{L^{1}([\tau,\infty))}$ are arbitrarily small provided $\tau$
is taken sufficiently large.

On the other hand, from continuity of the operator $K$ we conclude that $Kx_{n}$ converges to 
$Kx$ in $L^{1}$. Then, the sequence $\left\{Kx_{n}\right\}$ is relatively compact. In view
of Theorem \ref{bs} we infer that  the terms of the sequence $\left\{\left\|Kx_{n}\right\|_{L^{1}([\tau,\infty))}\right\}$
are arbitrarily small provided the number $\tau$ is large enough.

Keeping in mind the inequality (\ref{Uin}) from assumption (A5) and applying 
the so-called majorant principle ( see \cite{KZPS,ZKKM}), we conclude that the operator 
$U|_{L^{1}([0,\tau])}: L^{1}([0,\tau])\longrightarrow L^{1}([0,\tau])$ is continuous. Then, 
we deduce that $\delta_{n}$ converges to $0$ as $n$ goes to infinity.

 Therefore $Ux_{n}$ converges to $Ux$ in $L^{1}$. This means that the operator  $U$ maps $L^{1}(\mathbb{R}_{+})$  continuously into itself.
From this fact, assumption (A1), Theorem \ref{sup} and the continuity of $Q$
 we conclude that the operator $A=N_{f}UQ$ is continuous on the space $L^{1}$.

In the sequel, we prove the estimation (\ref{Aest}). For any nonempty measurable subset $I$ of $\mathbb{R}_{+}$ and $x\in  L^{1}$ we have
\begin{eqnarray*} \nonumber
\int_I|A x(t)|dt&\leq&\int_I a_{1}(t)dt+b_{1}\int_{I}\left|\int^{t}_{0}k(t,s)u(t,s,(Qx)(s)\,ds\right|dt\\
\nonumber\;&\leq&\int_I a_{1}(t)dt+b_{1}\int_{I}\left(\int^{t}_{0}\left|k(t,s)\right|\left(\alpha(s)+\beta\left|(Qx)(s)\right|\right)\,ds\right)dt\\
\nonumber\;&\leq&\int_I a_{1}(t)dt+b_{1}\int_{I}\left(\int^{t}_{0}\left|k(t,s)\right|\left(\alpha(s)+\beta\gamma_{2}(s)\right)\,ds\right)dt\\
\nonumber\;&&+b_{1}\beta\rho_{2}\int_{I}\left(\int^{t}_{0}\left|k(t,s)\right|\left(\left|x(\psi(s))\right|\right)\,ds\right)dt\\
\nonumber\;&=&\left\|a_{1}\right\|_{I}+b_{1}\left\|K\alpha\right\|_{I}+b_{1}\beta\left\|K\gamma_{2}\right\|_{I}+b_{1}\beta \rho_{2}\left\|Kx(\psi)\right\|_{I}\\
\nonumber\;&\leq&\left\|a_{1}\right\|_{I}+b_{1}\left\|K\right\|\left\|\alpha\right\|_{I}+b_{1}\beta\left\|K\right\|\left\|\gamma_{2}\right\|_{I}\\
\nonumber\;&&+b_{1}\beta \rho_{2}\left\|K\right\|\int_{I}\left|x(\psi(t))\right|\,dt\\
\nonumber\;&\leq&\left\|a_{1}\right\|_{I}+b_{1}\left\|K\right\|\left\|\alpha\right\|_{I}+b_{1}\beta\left\|K\right\|\left\|\gamma_{2}\right\|_{I}\\
\nonumber\;&&+b_{1}\beta \rho_{2}\left\|K\right\|M^{-1}\int_{I}\left|x(\psi(t))\right|\psi^{\prime}(t)\,dt\\
\nonumber\;&\leq&\left\|a_{1}\right\|_{I}+b_{1}\left\|K\right\|\left\|\alpha\right\|_{I}+b_{1}\beta\left\|K\right\|\left\|\gamma_{2}\right\|_{L^{1}(I)}\\
\nonumber\;&&+b_{1}\beta \rho_{2}M^{-1}\left\|K\right\|\int_{\psi(I)}x(\theta)\,d\theta
\end{eqnarray*}
Hence, we obtain the estimation (\ref{Aest}). In the same way,
the estimation (\ref{Best}) is  obtained using our assumptions on $g$ and $T$. \medskip
\end{proof}

Also, we need   the following lemma to prove our main result. In the proof of this lemma, we implement
a technique used in \cite{BaCh}.
\begin{lemma}\label{lws}
Let $Z$ be a nonempty, bounded and  relatively weakly compact set of $L^{1}(\mathbb{R}_{+})$
and $I_{\tau}=[0,\tau]$, where $\tau>0$. For any $\epsilon>0$ there exists a closed subset
$D_{\epsilon}$of the interval $I_{\tau}$ with $meas(I_{\tau}\backslash D_{\epsilon})\leq\epsilon$ such that the set
$A(Z)$ is relatively compact in the space $C(D_{\epsilon})$.
\end{lemma}
\begin{proof}[\bf Proof] Let $\epsilon>0$. In view of Theorem \ref{Cara}, we can find a closed 
subset $D_{\epsilon}$ of the interval $I_{\tau}=[0,\tau]$ such that 
$meas(D^{\prime}_{\epsilon})\leq\epsilon$ (where $D^{\prime}_{\epsilon}=I_{\tau}\backslash D_{\epsilon}$)
 and such that $f|_{D_{\epsilon}\times \mathbb{R}}$ and $k|_{D_{\epsilon}\times I_{\tau}}$ are continuous.
In the sequel, we show that $A(Z)$ is equibounded and equicontinuous in the space $C(D_{\epsilon})$
in order to apply Ascoli-Arz\'ela theorem. Let us take an arbitrary $x\in Z$.
Then for every $t\in D_{\epsilon}$, we have 

\begin{eqnarray*} 
\nonumber\left|(UQx)(t)\right|&=& \left|\int^{t}_{0}k(t,s)u(t,s,(Qx)(s))\,ds\right|\\
\nonumber\;&\leq&\int^{t}_{0}\left|k(t,s)\right|(\alpha(s)+\beta\left|(Qx)(s)\right|)\,ds\\
\nonumber\;&\leq&\int^{t}_{0}\left|k(t,s)\right|(\alpha(s)+\beta\gamma_{2}(s)+\beta\rho_{2}\left|x(\psi(s))\right|)\,ds\\
\nonumber\;&\leq& \overline{k}\left(\left\|\alpha\right\|+\beta\left\|\gamma_{2}\right\|+\beta\rho_{2}\int^{t}_{0}\left|x(\psi(s))\right|\,ds\right)\\
\nonumber\;&\leq& \overline{k}\left(\left\|\alpha\right\|+\beta\left\|\gamma_{2}\right\|+\beta\rho_{2} M^{-1}\int^{t}_{0}\left|x(\psi(s))\right|\psi^{\prime}(s)\,ds\right)\\
\nonumber\;&=& \overline{k}\left(\left\|\alpha\right\|+\beta\left\|\gamma_{2}\right\|+\beta\rho_{2} M^{-1}\int^{\psi(t)}_{\psi(0)}\left|x(\theta)\right|\,d\theta\right)\\
\nonumber\;&\leq& \overline{k}\left(\left\|\alpha\right\|+\beta\left\|\gamma_{2}\right\|+\beta\rho_{2} M^{-1}\left\|x\right\|\right)\\
\;&\leq&\overline{k}\left(\left\|\alpha\right\|+\beta\left\|\gamma_{2}\right\|+\beta\rho_{2}\,\overline{z} M^{-1}\right)
\end{eqnarray*}
 where $\overline{k}=\max\left\{\left|k(t,s)\right|: (t,s)\in D_{\epsilon}\times I_{\tau}\right\}$ and
$\overline{z}=\sup\left\{\left\|x\right\|: x\in Z\right\}$. In the sequel, we will denote by $U_{\epsilon}$ the quantity

\begin{equation}\label{uepsilon}
U_{\epsilon}:=\overline{k}\left(\left\|\alpha\right\|+\beta\left\|\gamma_{2}\right\|+\beta\rho_{2}\,\overline{z} M^{-1}\right).
\end{equation}
Then, using the assumption (A1) we obtain

$$\left|(Ax)(t)\right|\leq a_{1}(t)+b_{1}\left|(UQx)(t)\right|\leq \overline{a}_{1}+b_{1}U_{\epsilon}$$
for every $t\in D_{\epsilon}$, where $\overline{a}_{1}=\sup\left\{a_{1}(t): t\in D_{\epsilon}\right\}$.
This proves that the set $A(Z)$ is equibounded on the set $D_{\epsilon}$. \\
Now, let us consider $t_{1},t_{2}\in D_{\epsilon}$, $t_{1}\leq t_{2}$ and $\delta=t_{2}-t_{1}$. 
For a fixed $x\in Z$, we denote $\mathcal{U}^{t_{1},t_{2}}_{x}=(UQx)(t_{2})-(UQx)(t_{1})$. 
Then, in view of our assumptions, we have
\begin{eqnarray*} 
\nonumber\,\left|\mathcal{U}^{t_{1},t_{2}}_{x}\right|&=&\left|\int^{t_{2}}_{0}k(t_{2},s)u(t_{2},s,(Qx)(s))\,ds-\int^{t_{1}}_{0}k(t_{1},s)u(t_{1},s,(Qx)(s))\,ds\right|\\
\nonumber\;&\leq& \int^{t_{1}}_{0}\left|k(t_{2},s)u(t_{2},s,(Qx)(s))-k(t_{1},s)u(t_{1},s,(Qx)(s))\right|\,ds\\
\nonumber\;&& +\int^{t_{2}}_{t_{1}}\left|k(t_{2},s)u(t_{2},s,(Qx)(s))\right|\,ds\\
\nonumber\;&\leq& \int^{t_{1}}_{0}\left|k(t_{2},s)u(t_{2},s,(Qx)(s))-k(t_{1},s)u(t_{2},s,(Qx)(s))\right|\,ds\\
\nonumber\;&& + \int^{t_{1}}_{0}\left|k(t_{1},s)u(t_{2},s,(Qx)(s))-k(t_{1},s)u(t_{1},s,(Qx)(s))\right|\,ds\\
\nonumber\;&& +\int^{t_{2}}_{t_{1}}\left|k(t_{2},s)u(t_{2},s,(Qx)(s))\right|\,ds\\
\nonumber\;&=& \int^{t_{1}}_{0}\left|k(t_{2},s)-k(t_{1},s)\right|.\left|u(t_{2},s,(Qx)(s))\right|\,ds\\
\nonumber\;&& + \int^{t_{1}}_{0}\left|k(t_{1},s)\right|.\left|u(t_{2},s,(Qx)(s))-u(t_{1},s,(Qx)(s))\right|\,ds\\
\nonumber\;&& +\int^{t_{2}}_{t_{1}}\left|k(t_{2},s)\right|.\left|u(t_{2},s,(Qx)(s))\right|\,ds\\
\nonumber\;&\leq& \int^{t_{1}}_{0}\left|k(t_{2},s)-k(t_{1},s)\right|\left[\alpha(s)+\beta \gamma_{2}(s)+\beta\rho_{2}\left|x(\psi(s))\right|\right]\,ds\\
\nonumber\;&& + \int^{t_{1}}_{0}\left|k(t_{1},s)\right|h(\delta)\left[\gamma(s)+\lambda \gamma_{2}(s)+\lambda\rho_{2}\left|x(\psi(s))\right|\right]\,ds\\
\nonumber\;&& +\int^{t_{2}}_{t_{1}}\left|k(t_{2},s)\right|\left[\alpha(s)+\beta \gamma_{2}(s)+\beta\rho_{2}\left|x(\psi(s))\right|\right]\,ds,
\end{eqnarray*}
if $\delta$ is taken small enough.
 Now, we denote by $w^{\tau}(k,.)$  the modulus of continuity of the function $k$ on the set 
$D_{\epsilon}\times I_{\tau}$ given by

$$w^{\tau}(k,\delta)=\sup\left\{\left|k(t,s)-k(t,\sigma)\right|: t\in D_{\epsilon} \text{ and }s,\sigma \in I_{\tau} \text{ with }\left|s-\sigma\right|\leq\delta\right\}.$$
Therefore, we obtain
\begin{eqnarray*} \nonumber
\nonumber\;\left|\mathcal{U}^{t_{1},t_{2}}_{x}\right|&\leq& w^{\tau}(k,\delta)\left(\int^{\tau}_{0}\left[\alpha(s)+\beta \gamma_{2}(s)+\beta\rho_{2}\left|x(\psi(s))\right|\right]\,ds\right)\\
\nonumber\;&& + \,\overline{k}\,h(\delta)\int^{\tau}_{0}\left[\gamma(s)+\lambda \gamma_{2}(s)+\lambda\rho_{2}\left|x(\psi(s))\right|\right]\,ds\\
\nonumber\;&& + \,\overline{k}\int^{t_{2}}_{t_{1}}\left[\alpha(s)+\beta \gamma_{2}(s)+\beta\rho_{2}\left|x(\psi(s))\right|\right]\,ds\\
\nonumber\;&\leq& w^{\tau}(k,\delta)\left(\left\|\alpha\right\|+\beta \left\|\gamma_{2}\right\|+\beta\rho_{2}M^{-1}\int^{\tau}_{0}\left|x(\psi(s))\right|\psi^{\prime}(s)\,ds\right)\\
\nonumber\;&& + \,\overline{k}\,h(\delta)\left(\left\|\gamma\right\|+\lambda \left\|\gamma_{2}\right\|+\lambda\rho_{2}M^{-1}\int^{\tau}_{0}\left|x(\psi(s))\right|\psi^{\prime}(s)\,ds\right)\\
\nonumber\;&& + \,\overline{k}\left(\int^{t_{2}}_{t_{1}}\left(\alpha(s)+\beta \gamma_{2}(s)\right)\,ds+\beta\rho_{2}M^{-1}\int^{t_{2}}_{t_{1}}\left|x(\psi(s))\right|\psi^{\prime}(s)\,ds\right)\\
\nonumber\;&=& w^{\tau}(k,\delta)\left(\left\|\alpha\right\|+\beta \left\|\gamma_{2}\right\|+\beta\rho_{2}M^{-1}\int^{\psi(\tau)}_{0}\left|x(\theta)\right|\,d\theta\right)\\
\nonumber\;&& + \,\overline{k}\,h(\delta)\left(\left\|\gamma\right\|+\lambda \left\|\gamma_{2}\right\|+\lambda\rho_{2}M^{-1}\int^{\psi(\tau)}_{0}\left|x(\theta)\right|\,d\theta\right)\\
\nonumber\;&& + \,\overline{k}\left(\int^{t_{2}}_{t_{1}}\left(\alpha(s)+\beta \gamma_{2}(s)\right)\,ds+\beta\rho_{2}M^{-1}\int^{\psi(t_{2})}_{\psi(t_{1})}\left|x(\theta)\right|\,d\theta\right)\\
\nonumber\;&\leq& w^{\tau}(k,\delta)\left(\left\|\alpha\right\|+\beta \left\|\gamma_{2}\right\|+\beta\rho_{2}M^{-1}\overline{z}\right)\\
\nonumber\;&& + \,\overline{k}\,h(\delta)\left(\left\|\gamma\right\|+\lambda \left\|\gamma_{2}\right\|+\lambda\rho_{2}M^{-1}\overline{z}\right)\\
\nonumber\;&& + \,\overline{k}\left(\int^{t_{2}}_{t_{1}}\left(\alpha(s)+\beta \gamma_{2}(s)\right)\,ds+\beta\rho_{2}M^{-1}\int^{\psi(t_{2})}_{\psi(t_{1})}\left|x(\theta)\right|\,d\theta\right)
\end{eqnarray*} 
Keeping in mind that $k|_{D_{\epsilon}\times I_{\tau}}$ is uniformly continuous, we conclude that
$w^{\tau}(k,\delta)$ is arbitrarly small provided that the number $\delta$ is small enough.\\
Absolute continuity of $\psi$ ensures that $\psi(t_{2})-\psi(t_{1})$ is small enough when we take $\delta$
small enough. Considering the fact that $Z$ is bounded, relatively weakly compact 
and using Theorem \ref{bs} we obtain that the elements of the set 

$$\left\{\int^{\psi(t_{2})}_{\psi(t_{1})}\left|x(\theta)\right|\,d\theta: x\in Z\right\},$$  
are uniformly arbitrarily small provided the number $\delta$ is small enough. \\
Similarly, the number $\displaystyle\int^{t_{2}}_{t_{1}}(\alpha(s)+\beta \gamma_{2}(s))\,ds$ is arbitrarly small 
provided the number $\delta$ is small enough.\\
The number $h(\delta)$ is  arbitrarly small provided the number $\delta$ is small enough, thanks to hypothesis $\Lim_{\delta\rightarrow0}h(\delta)=0$ from assumption (A5).\\
This proves that the set $U(Z)$ is equicontinuous in the space $C(D_{\epsilon})$. Hence, uniform 
continuity of $f|_{D_{\epsilon}\times\left[0,U_{\epsilon}\right]}$, where $U_{\epsilon}$
is given by (\ref{uepsilon}), implies that the
set $A(Z)$ is equicontinuous in the space $C(D_{\epsilon})$.\\
Therefore, $A(Z)$ is equibounded and equicontinuous in the space $C(D_{\epsilon})$. Then, by Ascoli-Arz\'ela
theorem we obtain that  $A(Z)$ is a relatively compact set in the space $C(D_{\epsilon})$.
\end{proof}

\section{Proof of Theorem \ref{mth1}}\label{sfo}

\begin{proof}[\bf Proof of Theorem \ref{mth1}]
  We prove that the operators $A$ and $B$ from Eq. (\ref{e1re})
satisfy the 
hypothesis of  Theorem \ref{theow2}. \\
\textbf{Step 1:} First we prove that there exists a positive number $r>0$ such that 

$$A(B_{r})+B(B_{r})\subseteq B_{r}.$$ 
Let $x,y\in L^{1}=L^{1}(\R_{+})$. From estimations (\ref{Aest}) and (\ref{Best}), we get
\begin{eqnarray} \nonumber
\left\|A x+B y\right\|&\leq&\left\|a_{1}\right\|+b_{1}\left\|K\right\|\left[\left\|\alpha\right\|+\beta\left\|\gamma_{2}\right\|+\beta M^{-1}\rho_{2}\left\|x\right\|\right]\\
\nonumber\,&&+\left\|a\right\|+b\left[\left\|\gamma_{1}\right\|+\rho_{1}m^{-1}\left\|y\right\|\right]\\
 &\leq &C+b_{1}\rho_{2}\beta M^{-1}\left\|K\right\|\left\|x\right\|+b\rho_{1}m^{-1}\left\|y\right\|, \label{ABestxy}
\end{eqnarray}
where $C=\left\|a_{1}\right\|+\left\|a\right\|+b_{1}\left\|K\right\|(\left\|\alpha\right\|+\beta\left\|\gamma_{2}\right\|)+b\left\|\gamma_{1}\right\|$.\\
Let 

\begin{equation}\label{Ga}
\gamma=b\,\rho_{1}m^{-1}+b_{1}\rho_{2}\beta M^{-1} \left\|K\right\|,
\end{equation}
 and $r$ be the real number defined by
$r=\Frac{C}{1-\gamma}$. Thanks to hypothesis $(A7)$, we have $r>0$. Clearly if $x,y\in B_{r}$, then the estimation (\ref{ABestxy}) becomes

$$\left\|A x+B y\right\|\leq C+\gamma r=r.$$
 \textbf{Step 2:} Now, we show that there exists $\gamma \in [0,1[$ such that $\mu (AS+BS)\leq \gamma \mu(S)$ for every bounded subset $S$ of $L^1$, where $\mu$ is the measure of weak noncompactness defined by (\ref{mfe}). Let us fix a nonempty subset $S$ of $L^1$.\\  
 From estimations (\ref{Aest}) and (\ref{Best}), for every  nonempty measurable subset $I$ of $\mathbb{R}_+$  and for any $x\in S$  we have 
\begin{eqnarray} \nonumber
\left\|A x+B x\right\|_{I}&\leq&\left\|a_{1}\right\|_{I}+b_{1}\left\|K\right\|\left[\left\|\alpha\right\|_{I}+\beta\left\|\gamma_{2}\right\|_{I}+\beta M^{-1}\rho_{2}\left\|x\right\|_{\psi(I)}\right]\\
&&+\left\|a\right\|_{I}+b\left[\left\|\gamma_{1}\right\|_{I}+\rho_{1}m^{-1}\left\|x\right\|_{\phi(I)}\right], \label{ABestx}
\end{eqnarray}
Observe that the set consisting of one element of $L^1$ is weakly compact. Then, from  Theorem \ref{bs} we conclude that 

$$\displaystyle\lim_{\epsilon\rightarrow0}\left\{sup\left[\left\|h\right\|_{I}:meas(I)\leq\epsilon\right]\right\}=0,$$
for $h\in\left\{a_{1}, \gamma_{1}, \gamma_{2}, \alpha, a\right\}$.\\
Therefore, using the definition (\ref{cdef}) and taking into account that the functions
$\psi$ and $\phi$ are supposed to be absolutely continuous, from the inequality (\ref{ABestx}) we obtain

\begin{equation}\label{cest}
c(AS+BS)\leq \gamma c(S),
\end{equation} 
where $\gamma$ is given by (\ref{Ga}). From assumption $(A7)$ we have $\gamma<1$.

Now, consider an arbitrary $\tau>0$. Taking $I=[\tau,\infty)$, the inequality (\ref{ABestx}) becomes
\begin{eqnarray*} \nonumber
\int^{\infty}_{\tau}\left|(A x)(t)+(B x)(t)\right|\,dt&\leq& \int^{\infty}_{\tau}a_{1}(t)\,dt+\int^{\infty}_{\tau}a(t)\,dt\\
\nonumber\;&&+b_{1}\left\|K\right\|\left(\int^{\infty}_{\tau}\alpha(t)dt+\beta\,\int^{\infty}_{\tau}\gamma_{2}(t)\,dt\right)\\
\nonumber\;&&+b_{1}\rho_{2}\beta M^{-1}\left\|K\right\|\int^{\infty}_{\psi(\tau)}\left|x(t)\right|\,dt\\
\nonumber\;&&+b\int^{\infty}_{\tau}\gamma_{1}(t)\,dt+b\,\rho_{1}m^{-1}\int^{\infty}_{\phi(\tau)}\left|x(t)\right|\,dt.
\end{eqnarray*}
Theorem \ref{bs} ensures that $d(\left\{h\right\})=0$ for $h\in\left\{a_{1}, \gamma_{1}, \gamma_{2}, \alpha, a\right\}$, where $d$ 
is defined by (\ref{ddef}). Then, the last estimate  leads to

\begin{equation}\label{dest}
d(AS+BS)\leq \gamma d(S).
\end{equation}
 Combining estimations (\ref{cest}) and (\ref{dest}) with the definition (\ref{mfe}) we obtain

$$\mu(AS+BS)\leq \gamma \mu(S).$$
\textbf{Step 3:} $A$ is $ws$-compact. From Proposition \ref{Acont}, the operator $A$ is continuous.
 Now, consider a weakly convergent sequence $\left\{x_{n}\right\}$ in $B_{r}$ and fix
a number $\epsilon>0$. 
Applying Theorem \ref{bs} for the relatively compact set $\left\{Ax_{n}:n\in\mathbb{N}\right\}$, 
we deduce that there exist $\tau>0$  and $\delta>0$ such that for any $n\in\mathbb{N}$ we have 

\begin{equation}\label{tauAxn}
\int^{\infty}_{\tau}\left|(Ax_{n})(t)\right|\,dt\leq\frac{\epsilon}{8},
\end{equation}
and 

\begin{equation}\label{deltaAxn}
\int_{D}\left|(Ax_{n})(t)\right|\,dt\leq\frac{\epsilon}{4},
\end{equation}
for each  subset of $\mathbb{R}$ such that $meas(D)\leq\delta$.\\
By using Lemma \ref{lws} for $Z=\left\{x_{n}:n\in\mathbb{N}\right\}$, for every $p\in \mathbb{N}$, there exists 
a closed subset $D_{p}$ of the interval $I_{\tau}=[0,\tau]$ with 
$meas(D^{\prime}_{p})\leq \frac{1}{p}$ such that 
$\left\{Ax_{n}:n\in\mathbb{N}\right\}$  is relatively compact  in the space $C(D_{p})$.
Passing to subsequences if necessary we can assume that $\left\{Ax_{n}\right\}$
is a Cauchy sequence in  $C(D_{p})$, for every $p\in \mathbb{R}$. Then, we can choose 
$p_{0}$ large enough such that $meas(D^{\prime}_{p_{0}})\leq\delta$ and for every 
$m,n\geq p_{0}$ 

\begin{equation}\label{Amn}
\left\|Ax_{n}-Ax_{m}\right\|_{C(D_{p_{0}})}=
\max_{t\in D_{p_{0}}}\left|Ax_{n}(t)-Ax_{m}(t)\right|
\leq\frac{\epsilon}{4(meas(D^{\prime}_{p_{0}})+1)}.
\end{equation}
Now, using (\ref{deltaAxn}) with  (\ref{Amn}) we obtain
\begin{eqnarray} \nonumber
\int^{\tau}_{0}\left|Ax_{n}(t)-Ax_{m}(t)\right|\,dt&=&\int_{D_{p_{0}}}\left|Ax_{n}(t)-Ax_{m}(t)\right|\,dt\\
\nonumber\;&&+\int_{D^{\prime}_{p_{0}}}\left|Ax_{n}(t)-Ax_{m}(t)\right|\,dt\\
\nonumber\;&\leq &\frac{\epsilon \,meas(D^{\prime}_{p_{0}})}{4(meas(D^{\prime}_{p_{0}})+1)}+\int_{D^{\prime}_{p_{0}}}\left|Ax_{n}(t)\right|\,dt\\
\nonumber\;&&+\int_{D^{\prime}_{p_{0}}}\left|Ax_{m}(t)\right|\,dt \\
&\leq & \frac{3\epsilon}{4}.\label{lest}
\end{eqnarray}
Finally, by combining (\ref{tauAxn}) and (\ref{lest}) for $m,n\geq p_{0}$ we obtain
 \begin{eqnarray*} \nonumber
\left\|Ax_{n}-Ax_{m}\right\|_{L^{1}}&=&\int^{\infty}_{0}\left|Ax_{n}(t)-Ax_{m}(t)\right|\,dt\\
\nonumber\;&\leq&\int^{\tau}_{0}\left|Ax_{n}(t)-Ax_{m}(t)\right|\,dt+
\int^{\infty}_{\tau}\left|Ax_{n}(t)-Ax_{m}(t)\right|\,dt \\
\nonumber\;&\leq&\int^{\tau}_{0}\left|Ax_{n}(t)-Ax_{m}(t)\right|\,dt+
\int^{\infty}_{\tau}\left|Ax_{n}(t)\right|\,dt\\
\nonumber\;&&+\int^{\infty}_{\tau}\left|Ax_{m}(t)\right|\,dt\\ 
\nonumber\;&\leq& \epsilon.
\end{eqnarray*}
This proves that $\left\{Ax_{n}\right\}$ is a Cauchy sequence in the Banach space
$L^{1}(\mathbb{R}_{+})$. 
Then $\left\{Ax_{n}\right\}$ has a strongly convergent subsequence in $L^{1}(\mathbb{R}_{+})$.\\
Now, by applying Theorem \ref{theow2} with $\mathcal{M}=B_{r}$ we obtain the existence of an integrable solution for the problem (\ref{e1}).\medskip
\end{proof}
\section{Example}
Consider the mixed-type functional integral equation

\begin{align}
\nonumber x(t)=&\frac{t}{t^3+1}+\frac{1}{4}\ln\left[1+\left(\frac{x^{3}(2t)}{1+x^{2}(2t)}+e^{-t}\int^{\infty}_{0}e^{-\tau}\frac{x(\tau)}{1+x^{2}(\tau)}\,d\tau\right)^{2}\right]\\
&+\frac{1}{2}\arctan \left[\int_0^t(t+s)e^{-t}u(t,s,(Qx)(s))\,ds\right]^{2},\label{exa}
\end{align}
where

\begin{equation*}
u(t,s,x)=\frac{1+t+s}{2+(1+t+s)^{3}}+\frac{ts\left(ts+\sqrt{3}\sin\,x\right)x}{4(s+1)(t^{2}s^{2}+1)}, 
\end{equation*}
and 

\begin{equation*}
(Qx)(t)=\Frac{x^{2}(t)}{(1+\left|x(t)\right|)}\Int^{t}_{0}e^{-(t+\tau)}\frac{x(\tau)}{1+x^{2}(\tau)}\,d\tau.
\end{equation*}
The equation (\ref{exa}) is of the form (\ref{e1}) with 

\begin{equation*}
g(t,x)=\frac{t}{t^{3}+1}+\frac{1}{4}\ln(1+x^{2}), 
\end{equation*}
\begin{equation*}
f(t,x)=\frac{1}{2}\arctan x^{2}, 
\end{equation*}
\begin{equation*}
k(t,s)=(t+s)e^{-t}, 
\end{equation*}

\begin{equation*}
(Tx)(t)=\Frac{x^{3}(2t)}{1+x^{2}(2t)}+e^{-t}\Int^{\infty}_{0}e^{-\tau}\frac{x(\tau)}{1+x^{2}(\tau)}\,d\tau.
\end{equation*}
Next, we prove that assumptions $(A1)-(A7)$ are fulfilled.
\begin{enumerate}[label=(A\arabic*)]
\item Taking into account that $\arctan x^{2}\leq 2x$ for $x\geq0$ and $\ln(1+x^{2})\leq x$, 
it is easy to see that $g(t,x)$ and $f(t,x)$ satisfy assumption $(A1)$ with $a(t)=t/(t^{3}+1)$, $b=1/4$, $a_{1}(t)=0$ and $b_{1}=1$.
\item  For  $x\in L^{1}(\mathbb{R}_{+})$, it is easy to see the inequalities 

$$\left|(Tx)(t)\right|\leq e^{-t}+\left|x(2t)\right|\quad\text{and}\quad\left|(Qx)(t)\right|\leq\left|x(t)\right|.$$

We take 

$$\gamma_{1}(t)=e^{-t}, \rho_{1}=1,  \phi(t)=2t, m=2,$$
and 

$$\gamma_{2}(t)=0, \rho_{2}=1,  \psi(t)=t, M=1,$$
Now, we will prove that $Q$ is continuous on $L^{1}(\mathbb{R}_{+})$. 
Let $\left\{x_{n}\right\}$ be a sequence in $L^{1}(\mathbb{R}_{+})$ which converges in $L^{1}(\mathbb{R}_{+})$
to a function $x\in L^{1}(\mathbb{R}_{+})$. Denoting $\rho_{n}=\left\|Qx_{n}-Qx\right\|$, we have 
\begin{eqnarray*}
\rho_{n}&\leq&\int^{\infty}_{0}\left|\Frac{x_{n}^{2}(t)}{(1+\left|x_{n}(t)\right|)}
-\Frac{x^{2}(t)}{(1+\left|x(t)\right|)}\right|
\left(\Int^{t}_{0}e^{-(t+\tau)}\frac{\left|x_{n}(\tau)\right|}{1+x_{n}^{2}(\tau)}\,d\tau\right)\,dt\\
&&+\int^{\infty}_{0}\Frac{x^{2}(t)}{(1+\left|x(t)\right|)}
\left(\Int^{t}_{0}e^{-(t+\tau)}\left|\frac{x_{n}(\tau)}{1+x_{n}^{2}(\tau)}-
\frac{x(\tau)}{1+x^{2}(\tau)}\right|\,d\tau\right)\,dt\\
&\leq&\int^{\infty}_{0}\left|x_{n}(t)-x(t)\right|
\left(\Int^{\infty}_{0}e^{-\tau}\,d\tau\right)\,dt\\
&&+\int^{\infty}_{0}\left|x(t)\right|
\left(\Int^{t}_{0}\left|x_{n}(\tau)-x(\tau)\right|\,d\tau\right)\,dt\\
&\leq&\int^{\infty}_{0}\left|x_{n}(t)-x(t)\right|\,dt
+\int^{\infty}_{0}\left|x(t)\right|\left(\Int^{\infty}_{0}\left|x_{n}(\tau)-x(\tau)\right|\,d\tau\right)\,dt\\
&=& (1+\left\|x\right\|)\left\|x_{n}-x\right\|.
\end{eqnarray*}
This proves that $Q$ is continuous.
\item Let  $x,y\in L^{1}(\mathbb{R}_{+})$. Using the Mean Value Theorem, we have

\begin{align*}
\left|g(t,(Tx)(t))-g(t,(Ty)(t))\right|&=\frac{1}{4}\left|\ln\left(1+\left((Tx)(t)\right)^{2}\right)-\ln\left(1+\left((Ty)(t)\right)^{2}\right)\right|\\
&\leq \frac{1}{4}\left|(Tx)(t)-(Ty)(t)\right|\\
&\leq\frac{1}{4} \left|\frac{x^{3}(2t)}{1+x^{2}(2t)}-\frac{y^{3}(2t)}{1+y^{2}(2t)}\right|\\
&\phantom{{}\leq}+\frac{1}{4}e^{-t}\int^{\infty}_{0}\left|\frac{x(\tau)}{1+x^{2}(\tau)}-\frac{y(\tau)}{1+y^{2}(\tau)}\right|d\tau\\
&\leq\frac{1}{2}\left|x(2t)-y(2t)\right|+\frac{1}{4}e^{-t}\int^{\infty}_{0}\left|x(\tau)-y(\tau)\right|d\tau.
\end{align*}
Hence, we get

\begin{align*}
\left\|g(t,(Tx)(t))-g(t,(Ty)(t))\right\|&=\int^{\infty}_{0}\left|g(t,(Tx)(t))-g(t,(Ty)(t))\right|\,dt\\
&\leq\frac{1}{2}\int^{\infty}_{0}\left|x(2t)-y(2t)\right|\,dt\\
&\phantom{{}\leq}+\frac{1}{4}\left(\int^{\infty}_{0}e^{-t}dt\right)\int^{\infty}_{0}\left|x(\tau)-y(\tau)\right|\,d\tau\\
&= \frac{1}{2}  \left\|x-y\right\|.
\end{align*}
Therefore $B$ is a strict contraction.
\item Obviously, the function $u$ satisfies Carath\'eodory conditions.
\item Taking into account that the function $h_{1}(z)=\Frac{z}{2+z^{3}}$ 
is nonincreasing for $z\geq1$ and the maximum of the function 
$h_{2}(z)=\Frac{z^{2}+\sqrt{3}z}{1+z^{2}}$ is $3/2$,  for every $t,s\geq0$ and $x\in \R$ we have 

\begin{align*}
\left|u(t,s,x)\right|&\leq\frac{1+s}{2+(1+s)^{3}}+\frac{t^{2}s^{2}+\sqrt{3}\,ts}{4(s+1)(t^{2}s^{2}+1)}\left|x\right|\\
&\leq\frac{1+s}{2+(1+s)^{3}}+\frac{3}{8}\left|x\right|
\end{align*}
We take $\alpha(s)=\Frac{1+s}{2+(1+s)^{3}}$ and $\beta=3/8$.\\
Now, for every $t,w,s\geq0$ and $x\in \R$ it can be easily observed that 
\begin{align}
\nonumber\left|u(t,s,x)-u(w,s,x)\right|&\leq\left|\frac{1+t+s}{2+(1+t+s)^{3}}-\frac{1+w+s}{2+(1+w+s)^{3}}\right|\\
\nonumber&\phantom{{}\leq}+\frac{1}{4(s+1)}\left|\left(\frac{t^{2}s^{2}}{(t^{2}s^{2}+1)}-\frac{w^{2}s^{2}}{(w^{2}s^{2}+1)}\right)x\right|\\
&\phantom{{}\leq}+\frac{\sqrt{3}}{4(s+1)}\left|\left(\frac{ts}{(t^{2}s^{2}+1)}-\frac{ws}{(w^{2}s^{2}+1)}\right)x\sin x \right|\label{uwt}
\end{align}

Applying the Mean Value Theorem for the function $h_{1}(z)=\Frac{z}{2+z^{3}}$ 
between $z_{1}=1+t+s$ and $z_{2}=1+w+s$ and taking into account that 
$\left|h^{\prime}_{1}(z)\right|\leq\Frac{2}{2+z^{3}}\leq\frac{2}{2+(1+s)^{3}}$
for every $z\geq1+s$ we obtain 
$\left|\Frac{1+t+s}{2+(1+t+s)^{3}}-\Frac{1+w+s}{2+(1+w+s)^{3}}\right|\leq\Frac{2}{2+(1+s)^{3}}\left|t-w\right|$.
Now, substituting  $z=ts$ and 
$y=ws$ into the inequalities $\left|\Frac{z^{2}}{1+z^{2}}-\Frac{y^{2}}{1+y^{2}}\right|\leq \left|z-y\right|$ and  $\left|\Frac{z}{1+z^{2}}-\Frac{y}{1+y^{2}}\right|\leq \left|z-y\right|$  we obtain
  from (\ref{uwt}) the inequality
\begin{eqnarray*}
\left|u(t,s,x)-u(w,s,x)\right|&\leq&\frac{2}{2+(1+s)^{3}}\left|t-w\right|+\frac{\left|t-w\right|s}{4(s+1)}\left|x\right|\\
&&+\frac{\sqrt{3}\left|t-w\right|s}{4(s+1)}\left|x\right|\\
&\leq& \left|t-w\right|\left(\frac{2}{2+(1+s)^{3}}+\frac{1+\sqrt{3}}{4}\left|x\right|\right),
\end{eqnarray*}
We take $h(\delta)=\left|\delta\right|$, $\displaystyle\gamma_{0}(s)=\frac{2}{2+(1+s)^{3}}$ and $\lambda=(1+\sqrt{3})/4$.
\item Since the kernel $k(t,s)=(t+s)e^{-t}$ is nonnegative we have $(Kx)(t)=\Int_0^t k(t,s)x(s)\;ds$. It is proved 
in \cite{BaCh} that $\left\|K\right\|=2/\sqrt{e}$.

\item We have $b=1/4$, $b_{1}=\rho_{1}=\rho_{2}=M=1$, 
$\beta=3/8$, $m=2$ and  $\left\|K\right\|=2/\sqrt{e}$. Therefore 

$$\gamma=b\,\rho_{1}m^{-1}+b_{1}\rho_{2}\beta M^{-1} \left\|K\right\|=\frac{1}{8}+\frac{3}{4\sqrt{e}}<1.$$
\end{enumerate}
Therefore,  Theorem \ref{mth1} ensures that Eq. (\ref{e1}) has at least one solution in the space $L^{1}(\mathbb{R}_{+})$.

\section*{References}

\bibliography{mybiblio}

\end{document}